\documentclass[12pt, letterpaper]{article}

\usepackage[margin=2in]{geometry} 
\usepackage{setspace} 
\usepackage{times} 
\usepackage{hyperref} 

\usepackage{amssymb,amsthm,latexsym,amsmath,epsfig,pgf}
\newtheorem{observation}{Observation}[section]
\newtheorem{theorem}{Theorem}[section]
\newtheorem{lemma}{Lemma}[section]

\newtheorem{corollary}{Corollary}[section]
\newtheorem{definition}{Definition}[section]
\newtheorem{problem}{Problem}

\theoremstyle{remark}

\begin{document}

\begin{center}
\Large{\textbf{$D$-Antimagic Labelings of \\Oriented 2-Regular Graphs}}\\[1em] 
    \normalsize{
        Ahmad Muchlas Abrar$^{1}$, Rinovia Simanjuntak$^{2,*}$\\[0.5em]
        $^1$Doctoral Program in Mathematics,\\ Faculty of Mathematics and Natural Sciences, Institut Teknologi Bandung, Bandung, 40132, West Java, Indonesia \\
        Email: 30123302@mahasiswa.itb.ac.id,\\ ORCID: \href{https://orcid.org/0009-0009-3593-8393}{0009-0009-3593-8393}\\[0.5em]
        $^2$Combinatorial Mathematics Research Group,\\ Faculty of Mathematics and Natural Sciences, Institut Teknologi Bandung, Bandung, 40132, West Java, Indonesia\\
        Email: rino@itb.ac.id,\\ ORCID: \href{https://orcid.org/0000-0002-3224-2376}{0000-0002-3224-2376}\\[1em]
        \textit{*Corresponding Author}
    }
\end{center}

\section*{Abstract}

Given an oriented graph $\overrightarrow{G}$ and $D$ a distance set of $\overrightarrow{G}$, $\overrightarrow{G}$ is $D$-antimagic if there exists a bijective vertex labeling such that the sum of all labels of the $D$-out-neighbors of each vertex is distinct. 

This paper investigates $D$-antimagic labelings of 2-regular oriented graphs. We characterize $D$-antimagic oriented cycles, when $|D|=1$; $D$-antimagic unidirectional odd cycles, when $|D|=2$; and $D$-antimagic $\Theta$-oriented cycles. Finally, we characterize $D$-antimagic oriented 2-regular graphs, when $|D|=1$, and $D$-antimagic $\Theta$-oriented 2-regular graphs. 

\vspace{0.2cm}
\noindent\textbf{Keyword:} $D$-antimagic labeling, oriented graph, oriented 2-regular
\noindent\textbf{Mathematics Subject Classification:} 05C78

\section{Introduction}

Let $\overrightarrow{G}$ be an oriented graph with vertex set $V(\overrightarrow{G})$ and arc set $A(\overrightarrow{G})$. An arc from vertex $u$ to vertex $v$ is denoted by $(u, v)$. The neighbors of $u$ are classified as \textit{out-neighbors} (those to which $u$ has outgoing arcs) and \textit{in-neighbors} (those from which $u$ receives incoming arcs). The \textit{in-degree} of $u$  is the number of its in-neighbors and the \textit{out-degree} is the number of its out-neighbors, denoted by $d^-(u)$ and $d^+(u)$, respectively. A vertex with zero in-degree is a \textit{source}, and a vertex with zero out-degree is a \textit{sink}. 

The \textit{distance from $u$ to $v$}, denoted $d(u, v)$, is the length of the shortest directed path from $u$ to $v$, if such a path exists; otherwise, $d(u, v) = \infty$. The \textit{diameter} of $\overrightarrow{G}$, denoted $\mathrm{diam}(\overrightarrow{G})$, is the greatest finite distance between any pair of vertices in $\overrightarrow{G}$. The graph $\overrightarrow{G}$ is \textit{connected} if, for every pair of distinct vertices $u$ and $v$, at least one of $d(u, v)$ or $d(v, u)$ is finite. If both $d(u, v)$ and $d(v, u)$ are finite for all vertex pairs $u$ and $v$, $\overrightarrow{G}$ is \textit{strongly connected}; otherwise, it is \textit{weakly connected}. In such a case, the orientation of $\overrightarrow{G}$ is referred to as a strong and a weak orientation, respectively.

The concept of distance antimagic labeling was introduced by Kamatchi and Arumugam in 2013~\cite{Kamatchi-64-13}. For a simple, undirected graph $G = (V,E)$ of order $n$, a bijection $h: V(G) \to \{1, 2, \dots, n\}$ is a \textit{distance antimagic labeling} of $G$ if, for every vertex $u$, the \textit{vertex weight} $\omega(u) = \sum_{v \in N(u)} h(v)$ is distinct. If $G$ admits such a labeling, it is called \textit{distance antimagic}. Various graph classes, including paths $P_n$, cycles $C_n \ (n \neq 4)$, wheels $W_n \ (n \neq 4)$~\cite{Kamatchi-64-13}, hypercubes $Q_n \ (n \geq 4)$~\cite{Kamatchi cube}, and circulant graphs~\cite{sy2014distance}, as well as graphs formed by cartesian, strong, direct, lexicographic, corona, and join products~\cite{simanjuntakprod, Handa}, have been shown to be distance antimagic.

Simanjuntak \textit{et al.}~\cite{AAC} generalized the previous concept to $D$-antimagic labeling, defined as follows. For a graph $G$ and a nonempty set of distances $D \subseteq \{0, 1, 2, \dots, \mathrm{diam}(G)\}$, the \textit{$D$-neighborhood of $u$} is $N_D(u) = \{v \mid d(v, u) \in D\}$. A bijection $f: V(G) \to \{1, 2, \dots, n\}$ is a \textit{$D$-antimagic labeling} of $G$ if the \textit{$D$-weight} $\omega_D(u) = \sum_{v \in N_D(u)} f(v)$ is distinct for all $u \in V(G)$. If such a labeling exists, $G$ is \textit{$D$-antimagic}. The $D$-antimagic labeling is a generalization in the sense that the distance antimagic labeling is a special case where $D = \{1\}$.

In \cite{Abrarlinearforest}, we extended the concept of $D$-antimagic labeling to oriented graphs, as defined below:


\begin{definition}
Let $\overrightarrow{G}$ be an oriented graph and $D$ be a set of finite distances between two vertices in $\overrightarrow{G}$. A \textbf{$D$-antimagic labeling of $\overrightarrow{G}$} is a bijection $f: V(\overrightarrow{G}) \to \{1, 2, \dots, |V(\overrightarrow{G})|\}$ such that for any distinct vertices $u$ and $v$ holds $\omega_D(u) \neq \omega_D(v)$,  where $\omega_D(v) = \sum_{x \in N_D(v)} f(x)$ is the \textbf{$D$-weight of $v$} and $N_D(v) = \{y\in V(\overrightarrow{G}) \mid d(v, y) \in D\}$ is the \textbf{$D$-neighborhood of $v$}. In such a case, $\overrightarrow{G}$ is called \textbf{$D$-antimagic}.
\end{definition}

The following tools are used in this paper:
\begin{lemma}[Handshake Lemma for Oriented Graphs]~\cite{DB West} \label{lem:handshake}
If $\overrightarrow{G}$ is an oriented graph, then
$$\sum_{u \in V(\overrightarrow{G})} d^-(u) = \sum_{u \in V(\overrightarrow{G})} d^+(u) = |A(\overrightarrow{G})|.$$
\end{lemma}

\begin{lemma}\label{lem:all 0} \cite{Abrarlinearforest}
All graphs are $\{0\}$-antimagic.
\end{lemma}

\begin{theorem}~\label{thm:D* iff D} \cite{Abrarlinearforest}
Let $\overrightarrow{G}$ be a strongly connected graph of order $n$. If $D^* = \{0, 1, \dots, \mathrm{diam}(\overrightarrow{G})\} \setminus D$ then $\overrightarrow{G}$ is $D$-antimagic if and only if $\overrightarrow{G}$ is $D^*$-antimagic.
\end{theorem}

This paper focuses on $D$-antimagic oriented 2-regular graphs. Notice that Lemma \ref{lem:handshake} directly causes any oriented 2-regular graphs to have an equal number of sinks and sources.

\begin{theorem}\label{sink=source}
For any orientation of the cycle $C_n$, the number of sinks is equal to the number of sources.
\end{theorem}

Our study starts in Section~\ref{sec:cycle} where we investigate some properties of $D$-antimagic labeling of cycles. We also characterize the following: (1) $D$-antimagic oriented cycles, when $|D|=1$; (2) $D$-antimagic unidirectional odd cycles, when $|D|=2$; and (3) $D$-antimagic $\Theta$-oriented cycles. We conclude with Section \ref{sec:2reg}, where we characterize $D$-antimagic oriented 2-regular graphs, when $|D|=1$, and $D$-antimagic $\Theta$-oriented 2-regular graphs. 

\section{D-antimagic labeling of oriented cycles}\label{sec:cycle}

We begin by applying Theorem~\ref{thm:D* iff D} to oriented cycles and obtain the following:

\begin{lemma}
Any oriented cycle $\overrightarrow{C_n}$ is not $\{0, 1, \dots, n-1\}$-antimagic.
\end{lemma}

\subsection{Unidirectional and $\Theta$-oriented cycles}

Throughout this paper, we only consider two important orientations of a cycle. The first is 
the \textit{unidirectional $\overrightarrow{C_n}$} where the oriented cycle of order $n$ with the
vertex set $V(\overrightarrow{C_n}) = \{v_i \mid 1 \leq i \leq n\}$ has the arc set $A(\overrightarrow{C_n}) = \{(v_i, v_{i+1}) \mid 1 \leq i \leq n-1\} \cup \{(v_n, v_1)\}$. Clearly, it is the only strong orientation of a cycle and the diameter of the unidirectional $\overrightarrow{C_n}$ is $n-1$, and so 

\begin{observation}\label{ob:n-1 in D then C uni}
    If $(n-1) \in D$ and $\overrightarrow{C_n}$ is $D$-antimagic, then $\overrightarrow{C_n}$ is unidirectional.
\end{observation}

In the next lemma, we investigate distance sets in which the unidirectional cycles are not $D$-antimagic.

\begin{lemma}
If $x$ is a factor of $n$ and $D = \{(cx) \mod n \mid c \in \mathbb{N}\}$, then the unidirectional $\overrightarrow{C_n}$ is not $D$-antimagic.
\end{lemma}

\begin{proof}
Rewrite $D = \{x_c \mid x_c = (cx) \mod n, 1 \leq c \leq \frac{n}{x}\}$. 
Consider the $D$-neighborhoods of vertices $v_1$ and $v_{1+x}$ (index operations are in modulo $n$):
    \begin{align*}
        N_D(v_{1+x}) &= \{v_{(1+x)+x_1)}, v_{(1+x)+x_2)}, \dots, v_{(1+x)+x_{\frac{n}{x}})}\} \\
        &= \{v_{1+x+x}, v_{1+x+2x}, \dots, v_{1+x+x(\frac{n}{x})}\} \\
        &= \{v_{(1+x_2}, v_{1+x_3}, \dots, v_{1+x_{\frac{n}{x}}}, v_{1+x_1}\} \\
        &= N_D(v_1).
    \end{align*}
For any vertex labeling, the $D$-weights of $v_1$ and $v_{1+x}$ are identical.
\end{proof}

Let $D = \{d_1, d_2, \dots, d_t\}$, where $t \neq 0$. For an integer $k$, define $D + k = \{(d_i + k) \mod n \mid 1 \leq i \leq t\}$ . The following theorem shows how the unidirectional cycle admits many $D$-antimagic labelings.

\begin{theorem}\label{th:D iff D+k}
    A unidirectional $\overrightarrow{C_n}$ is $D$-antimagic if and only if $\overrightarrow{C_n}$ is $(D + k)$-antimagic.
\end{theorem}

\begin{proof}
Let $D = \{d_1, d_2, \dots, d_t\}$, where $t \neq 0$, and $h$ be a $D$-antimagic labeling of the unidirectional $\overrightarrow{C_n}$. 
Consider:
    \begin{align*}
        N_{D+k}(v_i) =& \{v_{i + (d_j + k)} \mid 1 \leq j \leq t\} = \{v_{(i + k) + d_j} \mid 1 \leq j \leq t\}\\ =& N_D(v_{i + k}).
    \end{align*}
This establishes a one-to-one correspondence between $N_{D+k}(v_i)$ and $N_D(v_{i+k})$. Therefore, under $h$, $\omega_D(v_{i+k}) = \omega_{D+k}(v_i)$, and $\omega_{D+k}$ remains distinct for each vertex.
\end{proof}


The second orientation under consideration is the \textit{$\Theta$-orientation}, which is defined as an orientation such that the resulting oriented cycle has exactly one pair of adjacent sink and source. Without loss of generality, let a $\Theta$-oriented $\overrightarrow{C_n}$ has the arc set $\{(v_i,v_{i+1})|1\le i\le n-1\}\cup \{(v_1,v_n)\}$. In this case, each vertex has a unique out-neighbor, except for $v_1$ that has two out-neighbors. 

Now we are ready to characterize $\{1\}$-antimagic oriented cycles, where $|D|=1$.  
\begin{theorem}
\label{lem:1-A iff Theta1,2}
An oriented cycle $\overrightarrow{C_n}$ is $\{1\}$-antimagic if and only if it is unidirectional or $\Theta$-oriented.    
\end{theorem}
\begin{proof}
If $\overrightarrow{C_n}$ is neither unidirectional nor $\Theta$-oriented, then the orientation of $\overrightarrow{C_n}$ satisfies one of the following conditions: 
\begin{itemize}
    \item \textbf{There are more than one sink,} then $\overrightarrow{C_n}$ is trivially not $\{1\}$-antimagic.
    \item \textbf{The sink and the source are not neighbors,} then the sink will have two in-neighbors of out-degree one. Both share the sink as the only out-neighbor, thus $\overrightarrow{C_n}$ is not $\{1\}$-antimagic.
\end{itemize}
For necessity, if $\overrightarrow{C_n}$ is unidirectional, each vertex has a unique out-neighbor. In this case, any bijective vertex labeling is $\{1\}$-antimagic labeling. If $\overrightarrow{C_n}$ id $\Theta$-oriented, define a bijection $f$ as $f(v_i)=n-i+1$ for all $i=1,2,\dots,n$. In this case, each vertex has a distinct $\{1\}$-weight, where $\omega_{\{1\}}(v_1)=n$ and $\omega_{\{1\}}(v_i)=n-i$ for $i=2,3,\dots,n$. \end{proof}

The reasoning in the proof of Lemma \ref{lem:1-A iff Theta1,2} could be extended to study the orientations of $\overrightarrow{C_n}$ given the  minimum distance in $D$. 


\begin{lemma} \label{lem:min1}
If $\min(D)= 1$ and $\overrightarrow{C_n}$ is $D$-antimagic then $\overrightarrow{C_n}$ is either unidirectional or $\Theta$-oriented.
\end{lemma}

\begin{lemma} \label{lem:min2}
Let $\overrightarrow{C_n}$ be $D$-antimagic graph. If $\min(D)\ge 2$ then $\overrightarrow{C_n}$ is unidirectional.
\end{lemma}

From Lemmas \ref{lem:min1} and \ref{lem:min2}, we can conclude the following.  

\begin{lemma} \label{lem:min0}
Let $\overrightarrow{C_n}$ be $D$-antimagic.
If the orientation of $\overrightarrow{C_n}$ is neither unidirectional nor $\Theta$ then $\min(D)=0$.
\end{lemma}

However, we could not identify all possible orientations for a $D$-antimagic oriented cycle when $\min(D)=0$. Thus, the following open problem:
\begin{problem}
For $\min(D)=0$, list all possible orientations such that the oriented cycle $\overrightarrow{C_n}$ is $D$-antimagic. 
\end{problem}



Since a cycle with orientation \(\Theta\) has exactly one source and one sink that are adjacent, it is weakly connected with diameter \(n-2\). In fact, \(\Theta\) is the only weak orientation of a cycle with diameter \(n-2\).

\begin{lemma}\label{lem:n-2 uni or Theta}
Let \((n-2) \in D\). If \(\overrightarrow{C_n}\) is \(D\)-antimagic, then \(\overrightarrow{C_n}\) is either unidirectional or \(\Theta\)-oriented.
\end{lemma}


The next theorem provides all possible $D$ so that a \(\Theta\)-oriented cycle has a \(D\)-antimagic labeling. 

\begin{theorem}\label{the:all D for theta}
Let $n\geq 3$.  A $\Theta$-oriented $\overrightarrow{C_n}$ is $D$-antimagic if and only if $\min(D)\le 1$.  
\end{theorem}\vspace{-1em}
\begin{proof} Based on Lemma \ref{lem:min2}, we know that if a $\Theta$-oriented $\overrightarrow{C_n}$ is $D$-antimagic then $\min(D)\le 1$. 

From Lemma \ref{lem:all 0} and Theorem \ref{lem:1-A iff Theta1,2}, $\overrightarrow{C_n}$ is $\{0\}$- and $\{1\}$-antimagic. Thus we only need to proceed with $|D|\ge 2$. Let $D=\{d_0, d_1,\ldots, d_p\}$, where $p \geq 2$ and $0\le d_0<d_1<\ldots<d_p\le n-2$. Define a bijection $g:V(\overrightarrow{C_n})\rightarrow\{1, 2, \ldots, n\}$, with $g(v_i)=n-i+1$, for $i=1, 2, \ldots, n$. We shall prove that $g$ is $D$-antimagic, by considering two cases based on  $\min{(D)}=d_0$.
    
\noindent \textbf{Case 1 : $d_0=0$.} The $D$-neighborhoods and $D$-weights of $v_i$ in $\overrightarrow{C_n}$ are:
\begin{itemize}
\item \textbf{For $i=1$:} If $d_1\ne 1$ then $ N_D(v_1) = \{v_{1+d_t}|0 \le t \le p\}$ and $\omega_D(v_i)
=(p+1)(n-i+1)-\sum_{t=1}^p d_t$; if $d_1=1, N_D(v_1) = \{v_{1+d_t}|0 \le t \le p\}\cup\{v_n\}$ and $\omega_D(v_i)
=(p+1)(n-i+1)-\sum_{t=2}^p d_t$.
\item \textbf{For $2 \le i \le n-d_p$:} $N_D(v_i) = \{v_{i+d_t}|0 \le t \le p\}$ and $\omega_D(v_i)
=(p+1)(n-i+1)-\sum_{t=1}^pd_t$.
\item \textbf{For $n-d_{t_0+1}+1 \le i \le n-d_{t_0}\text{ and }0\le t_0\le p-1$,}   $N_D(v_i) = \{v_{i+d_t}|0 \le t \le t_0\}$ and $\omega_D(v_i)
=(t_0+1)(n-i+1)-\sum_{t=1}^{t_0} d_t$.
\end{itemize} 
\noindent \textbf{Case 2 : $d_0=1$.} The $D$-neighborhoods and $D$-weights of $v_i$ are as follows:
\begin{itemize}
\item  \textbf{For $i=1$,} $N_D(v_i) = \{v_{1+d_t}|0 \le t \le d_p\}\cup\{v_n\}$ and $\omega_D(v_1)=(p+1)n-\sum_{t=1}^pd_t$.\vspace{-0.5em}
\item \textbf{For $2 \le i \le n-d_p$,} $N_D(v_i) = \{v_{i+d_t}|0 \le t \le p\}$ and $\omega_D(v_i)=(p+1)(n-i+1)-\sum_{t=0}^pd_t$.  \vspace{-0.5em}
\item \textbf{For $n-d_{t_0+1}+1\le i \le n-d_{t_0}, 0\le t_0\le p-1$,} $N_D(v_i) = \{v_{i+d_t}|0 \le t \le t_0\}$ and $\omega_D(v_i)=(t_0+1)(n-i+1)-\sum_{t=0}^{t_0} d_t$.\vspace{-0.5em} 
\item \textbf{For $i=n$,} $N_D(v_n) = \emptyset$ and $\omega_D(v_n)=0$.
\end{itemize}
In both cases, $\omega_D(v_i)>\omega_D(v_{i+1})$, for $1\le i\le n-1$. 

Examples of the labelings are provided in Figure \ref{fig:C10 thets 0,1,3,4 dan 1,3,4,6}.\end{proof}

\begin{figure}[ht]
    \centering    \includegraphics[width=0.8\linewidth]{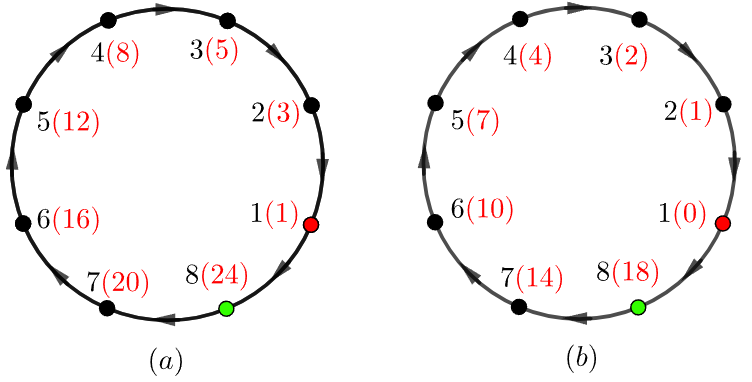}
    \caption{(a) $\{0,1,3,4\}$- and (b) $\{1,3,4,6\}$-antimagic labelings of $\Theta$ oriented of $\overrightarrow{C_{10}}$.}
    \label{fig:C10 thets 0,1,3,4 dan 1,3,4,6}
\end{figure}

\subsection{$D$-antimagic unidirectional cycles, with $|D|=1$} \label{sub:D=1}

Recall that Theorem \ref{lem:1-A iff Theta1,2} characterizes $\{1\}$-antimagic oriented cycles.  The next theorem completes our characterization of $D$-antimagic oriented cycles, with $|D|=1$.

\begin{theorem}\label{lem:uniC iff k-Ant}
For $2 \leq k \leq n-1$, the oriented cycle $\overrightarrow{C_n}$ is $\{k\}$-antimagic if and only if it is unidirectional.
\end{theorem}
\begin{proof} By Lemma \ref{lem:min2}, we only need to construct a $\{k\}$-antimagic labeling for a unidirectional cycle. Consider a bijection $f(v_i) = i$ for $1 \leq i \leq n$. Since each vertex has a unique $\{k\}$-neighbor given by $N_{\{k\}}(v_i) = \{v_{(i+k) \mod n}\}$, 
then each vertex has a distinct $\{k\}$-weight.\end{proof}

Applying Theorem \ref{thm:D* iff D} to Theorems \ref{lem:1-A iff Theta1,2}, \ref{lem:all 0}, and the previous theorem, we obtain the following:  

\begin{corollary}
For $0 \leq k \leq n-1$, the unidirectional cycle $\overrightarrow{C_n}$ is $\{0, 1, \dots, n-1\} \setminus \{k\}$-antimagic.   
\end{corollary}

\subsection{$D$-antimagic unidirectional cycles, with $|D|=2$} \label{sub:D=2}

After characterizing $D$-antimagic oriented cycles for $|D|=1$, the next step is to consider distance sets of cardinality at least 2. However, for $|D|=2$, we only managed to characterize $D$-antimagic labelings for unidirectional odd cycles. 

We start by considering the $D$-antimagic unidirectional cycles with $D=\{0,k\}$, where $1\leq k\leq 3$.

\begin{lemma} \label{lem:{0,1}cycle}
For $n \geq 3$, the unidirectional cycle $\overrightarrow{C_n}$ is $\{0,1\}$-antimagic.
\end{lemma}
\begin{proof}
In a unidirectional cycle $\overrightarrow{C_n}$, the $\{0,1\}$-neighborhood of each vertex are $N_{\{0,1\}}(v_n) = \{v_n, v_1\}$ and $N_{\{0,1\}}(v_i) = \{v_i, v_{i+1}\}$, for $1 \leq i \leq n-1$.

We consider two cases based on the parity of $n$.

\textbf{Case 1: $n \geq 3$ is odd}, assign a bijection $f: V(\overrightarrow{C_n}) \to \{1,2,\dots,n\}$ such that $f(v_i) = i$, for $1\le i\le n$. The $\{0,1\}$-weights of vertices are 
$\omega_{\{0,1\}}(v_i)= 2i + 1, \text{ for }1 \leq i \leq n-1 \text{ and }  \omega_{\{0,1\}}(v_n) = n + 1$.
   

\textbf{Case 2: $n \geq 4$ is even}, define a bijection $f_1$ such that $f_1(v_1) = 2$, $f_1(v_2) = 1$, and $f_1(v_i) = i$, for $3 \leq i \leq n$. The distinct $\{0,1\}$-weights are    
$\omega_{\{0,1\}}(v_1)  = 3,
\omega_{\{0,1\}}(v_2) = 4, 
\omega_{\{0,1\}}(v_i)= 2i + 1 \text{ for } 3 \leq i \leq n-1,\text{ and }
\omega_{\{0,1\}}(v_n)  = n + 2$. 

Figure \ref{fig:uni C8 0,1} gives examples of these labelings.\end{proof}



\begin{figure}[h]
    \centering    \includegraphics[width=0.5\linewidth]{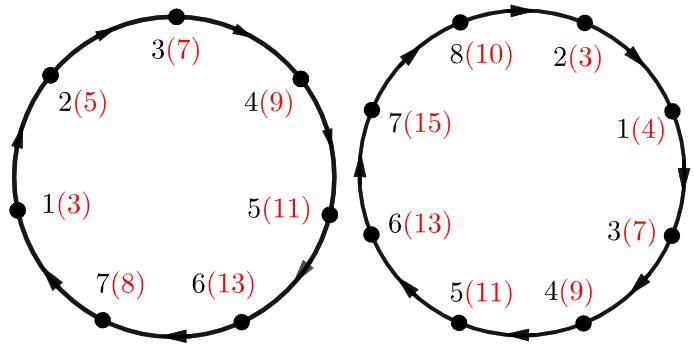}
    \caption{$\{0,1\}$-antimagic labelings of  unidirectional $\overrightarrow{C_7}$ and $\overrightarrow{C_8}$.}
    \label{fig:uni C8 0,1}
\end{figure}


\begin{lemma} \label{lem:{0,2}cycle}
For $n \geq 3$ and $n \neq 4$, the unidirectional  $\overrightarrow{C_n}$ is $\{0,2\}$-antimagic.
\end{lemma}
\begin{proof}
    The $\{0,2\}$-neighborhood of each vertex in $\overrightarrow{C_n}$ is as follows:
    \begin{align*}
        N_{\{0,2\}}(v_i) &= \{v_i, v_{i+2}\}, \quad \text{ for } 1 \leq i \leq n-2,\\
        N_{\{0,2\}}(v_{n-1}) &= \{v_{n-1}, v_1\}, \text{ and}\\
        N_{\{0,2\}}(v_n) &= \{v_n, v_2\}.
    \end{align*}

    To construct the $\{0,2\}$-antimagic labeling, we consider the following two cases: 

    \textbf{Case 1: $n \geq 3$ is odd}, define a bijection $f: V(\overrightarrow{C_n}) \to \{1, 2, \dots, n\}$ such that $f(v_i) = i$, for $1\le i\le n$. Based on $f$ and  the $\{0, 2\}$-neighborhoods, each vertex has distinct $\{0, 2\}$-weight where     
       $ \omega_{\{0,2\}}(v_i)=2i + 2, \text{ for } 1 \leq i \leq n-2,
        \omega_{\{0,2\}}(v_{n-1})=n,\text{ and }
        \omega_{\{0,2\}}(v_n) =n + 2.$


    \textbf{Case 2: $n \geq 4$ is even}, define a bijection $f'$ by $f'(v_1) = n$, $f'(v_n) = 1$, and $f'(v_i) = i$, for $2 \leq i \leq n-1$. So, the $\{0, 2\}$-weights are differs for all vertices, given by:
    \begin{align*}
        \omega_{\{0,2\}}(v_1) &= n + 3,\\
        \omega_{\{0,2\}}(v_n) &= 3,\\
        \omega_{\{0,2\}}(v_i) &= 2i + 2, \text{ for } 2 \leq i \leq n-3,\\
        \omega_{\{0,2\}}(v_{n-2}) &= n-1, \text{and}\\
        \omega_{\{0,2\}}(v_{n-1}) &= 2n-1.\end{align*}
        
Examples of the labelings are given in Figure \ref{fig:C7,8 uni 0,2}.\end{proof}



\begin{figure}[ht]
    \centering    \includegraphics[width=0.65\linewidth]{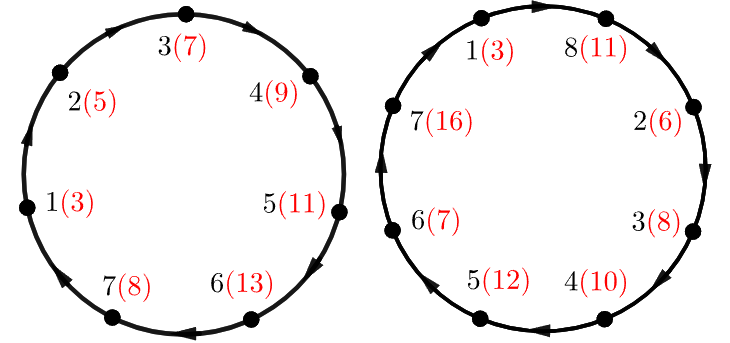}
    \caption{$\{0,2\}$-antimagic labelings of unidirectional $\overrightarrow{C_7}$ and $\overrightarrow{{C_8}}$.}
    \label{fig:C7,8 uni 0,2}
\end{figure}

\begin{lemma} \label{lem:{0,3}cycle}
For $n \geq 4$ and $n \neq 6$, the unidirectional $\overrightarrow{C_n}$ is $\{0,3\}$-antimagic.
\end{lemma}
\begin{proof}
Based on the orientation of unidirectional $\overrightarrow{C_n}$, the $\{0,3\}$-neighborhoods of each vertex are:\vspace{-1em}
    \begin{align*}
        N_{\{0,3\}}(v_i) &= \{v_i, v_{i+3}\}, \text{ for } 1 \leq i \leq n-3, \\
        N_{\{0,3\}}(v_{n-2}) &= \{v_{n-2}, v_{1}\}, \\
        N_{\{0,3\}}(v_{n-1}) &= \{v_{n-1}, v_{2}\}, \text{ and} \\
        N_{\{0,3\}}(v_n) &= \{v_n, v_{3}\}.
    \end{align*}

    This proof will be divided into two cases: 

    \textbf{Case 1: $n \geq 5$ is odd.}  
    Let $f: V(\overrightarrow{C_n}) \to \{1, 2, \dots, n\}$ be a bijection such that $f(v_i) = i$, for $1\le i\le n$. Then all $\{0,3\}$-weights are unique and given by:\vspace{-1em}
    \begin{align*}
        \omega_{\{0,3\}}(v_i) &= 2i + 3, \text { for }1 \leq i \leq n-3, \\
        \omega_{\{0,3\}}(v_{n-2}) &= n-1, \\
        \omega_{\{0,3\}}(v_{n-1}) &= n+1, \text{ and} \\
        \omega_{\{0,3\}}(v_n)&= n + 3.\end{align*}


    
    \textbf{Case 2: $n \geq 4$ is even.}  
    In $\overrightarrow{C_6}$, the pair $v_1$ and $v_4$ has the same $\{0,3\}$-neighborhood. Consequently, their $\{0,3\}$-weights are equal for any bijection, implying that $\overrightarrow{C_6}$ is not $\{0,3\}$-antimagic. 
    
    For $n \neq 6$, we will construct a $\{0,3\}$-antimagic labeling of $\overrightarrow{C_n}$. For $n=4$, label $v_1,v_2,v_3,$ and $v_4$ with $1,3,4,$ and $2$, respectively. In this case, $\omega_{\{0,3\}}(v_1)=3,\omega_{\{0,3\}}(v_2)=4,\omega_{\{0,3\}}(v_3)=7,$ and $\omega_{\{0,3\}}(v_4)=6$.
    For $n=8,10,12$, see the labelings in Figure \ref{fig:0,3 n=8,10,12}. 

    For even $n\ge 14$, let $f_1$ be a bijection define as follows:
   
    \begin{itemize}
    \item For $1\le i\le n-8$, $f_1(v_i)=i$, with $\omega_{\{0,3\}}(v_i)=2i+3$,
    \item $f_1(v_{n-7})=n-7,\quad \text{ with }\omega_{\{0,3\}}(v_{n-7})=2n-7,$
    \item $f_1(v_{n-6})=n-6,\quad \text{ with }\omega_{\{0,3\}}(v_{n-6})=2n-10,$
    \item $f_1(v_{n-5})=n-5,\quad \text{ with }\omega_{\{0,3\}}(v_{n-5})=2n-8,$
    \item $f_1(v_{n-4})=n,\textcolor{white}{.. .}\quad\quad \text{ with }\omega_{\{0,3\}}(v_{n-4})=2n-2,$
    \item $f_1(v_{n-3})=n-4,\quad \text{ with }\omega_{\{0,3\}}(v_{n-3})=2n-5,$
    \item $f_1(v_{n-2})=n-3,\quad \text{ with }\omega_{\{0,3\}}(v_{n-2})=n-2,$
    \item $f_1(v_{n-1})=n-2,\quad \text{ with }\omega_{\{0,3\}}(v_{n-1})=n$, and
    \item $f_1(v_n)=n-1,\quad \quad \text{ with }\omega_{\{0,3\}}(v_n)=n+2.$
    \end{itemize}
    By analyzing all cases, we conclude that the unidirectional $\overrightarrow{C_n}$, $n \geq 4$ and $n \neq 6$, is $\{0,3\}$-antimagic.
    \end{proof}
    
    \begin{figure}[ht]
        \centering        \includegraphics[width=1\linewidth]{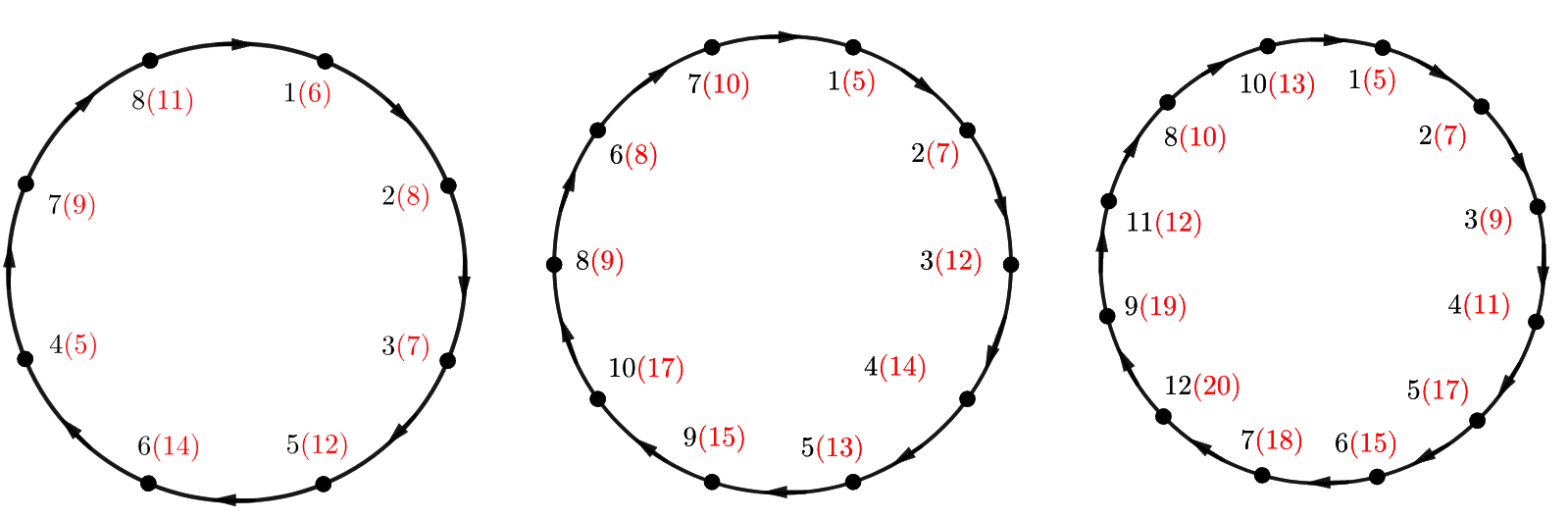}
        \caption{$\{0,3\}$-antimagic labelings of $\overrightarrow{C_8}, \overrightarrow{C_{10}}$, and $\overrightarrow{C_{12}}$.}
        \label{fig:0,3 n=8,10,12}
    \end{figure}

Applying Theorems \ref{th:D iff D+k} and \ref{thm:D* iff D} to Lemmas \ref{lem:{0,1}cycle}, \ref{lem:{0,2}cycle}, and \ref{lem:{0,3}cycle}, we obtain the following:  

\begin{corollary}\label{cor:all results uniC}
Let $a$ be a positive integer, where $0 \leq a \leq n-1$. For the following distance sets \(D\), the unidirectional \(\overrightarrow{C_n}\) is \(D\)-antimagic.
\begin{enumerate}
\item \(D = \{a, (a+1) \mod{n}\}\),
\item \(D = \{a, (a+2) \mod{n}\}\), for \(n \geq 3, n \neq 4\),
\item \(D = \{a, (a+3) \mod{n}\}\), for \(n \geq 4, n \neq 6\).
\item \(D = \{0, 1, \dots, n-1\} \setminus \{a, (a+1) \mod{n}\}\), 
\item \(D = \{0, 1, \dots, n-1\} \setminus \{a, (a+2) \mod{n}\},\text{ for } n \geq 3, n \neq 4.\)
\item \(D = \{0, 1, \dots, n-1\} \setminus \{a, (a+3) \mod{n}\},\text{ for } n \geq 4, n \neq 6.\)
\end{enumerate}
\end{corollary}

Due to Theorem
\ref{th:D iff D+k}, all distance sets $D$ of cardinality 2 of a $D$-antimagic unidirectional cycle can be represented by distance sets $\{0,k\}$, $1\le k\le \lfloor \frac{n}{2}\rfloor$. Thus, in the next theorem, we construct such labelings, however only for odd unidirectional cycles.

\begin{theorem}
For odd $n\ge 3$ and $|D|=2$, the unidirectional $\overrightarrow{C_n}$ is $D$-antimagic.  
\end{theorem}
\begin{proof} Let $D=\{0,k\}, 1\le k\le \lfloor \frac{n}{2}\rfloor$. Define a bijection $g(v_i)=i$, $i=1,2,\dots,n$. We consider two cases: 

\textbf{Case 1: $k$ is odd.} 
For $1\le i\le n-k$, $N_D(v_i)=\{v_i,v_{i+k}\}$ and $\omega_D(v_i)=2i+k$, where each $D$-weight is odd and distinct. For $n-k+1\le i\le n$, $N_D(v_i)=\{v_i,v_{i-n+k}\}$ and $\omega_D(v_i)=2i-n+k$, where each $D$-weight is even and distinct.

\textbf{Case 2: $k$ is even.} 
For $1\le i\le n-k$, $N_D(v_i)=\{v_i,v_{i+k}\}$ and $\omega_D(v_i)=2i+k$, where each $D$-weight is even and distinct. For $n-k+1\le i\le n$, $N_D(v_i)=\{v_i,v_{i-n+k}\}$ and $\omega_D(v_i)=2i-n+k$, where each $D$-weight is odd and distinct. 
    
Utilizing Theorem
\ref{th:D iff D+k}, the distance sets $D=\{0,k\}$ can be extended to all distance sets of cardinality 2. 

See Figure \ref{fig:Cn ganjil} for examples of the labelings.
\end{proof}

\begin{figure}[ht]
    \centering    \includegraphics[width=0.65\linewidth]{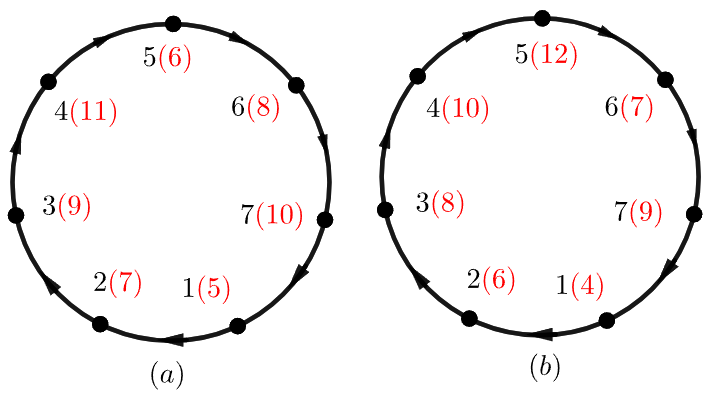}
    \caption{$\{0,k\}$-antimagic labelings of unidirectional $\overrightarrow{C_7}$ where (a) $k=4$ and (b) $k=3$.}
    \label{fig:Cn ganjil}
\end{figure}

Applying Theorem \ref{thm:D* iff D} to the previous theorem, we obtain the following:

\begin{corollary}
For odd $n\ge 3$ and $|D|=2$, the unidirectional $\overrightarrow{C_n}$ is $\{0, 1, \dots, n-1\} \setminus D$-antimagic.  
\end{corollary}

So far, we only succeeded in finding labelings for unidirectional odd cycles, and thus we pose the following question.

\begin{problem}
For $|D|=2$, determine whether unidirectional even cycles are $D$-antimagic.    
\end{problem}

\begin{figure}[ht]
    \centering    \includegraphics[width=0.9\linewidth]{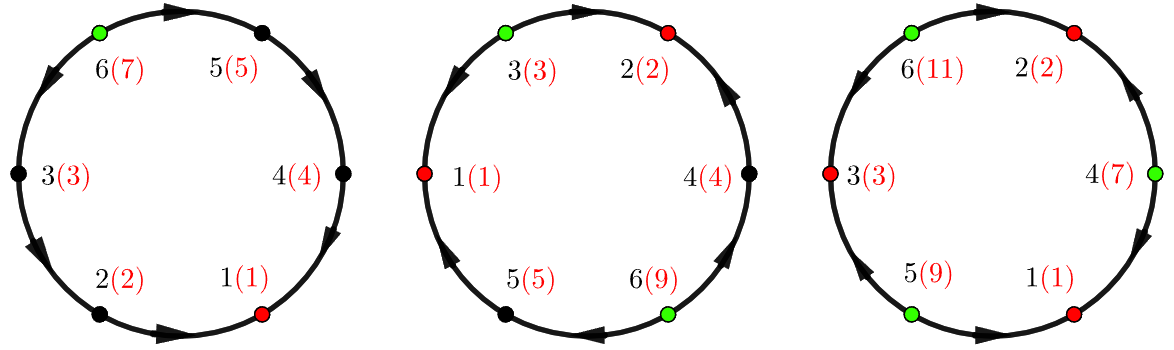}
    \caption{Oriented $\overrightarrow{C_6}$s that are $\{0,3\}$-, $\{0,2\}$-, and $\{0,1\}$-antimagic, respectively.}
    \label{fig:C6 D=2}
\end{figure}

Although for distance sets of cardinality 2, we only consider unidirectional cycles, we have examples in Figure \ref{fig:C6 D=2} where we provide non-unidirectional oriented $\overrightarrow{C_6}$s that are $\{0,3\}$-, $\{0,2\}$-, and $\{0,1\}$-antimagic. 
This leads to the following open problem.
\begin{problem}
For $|D|=2$, list all orientations of a cycle to be $D$-antimagic.     
\end{problem}

\begin{figure}[ht]
    \centering    \includegraphics[width=0.65\linewidth]{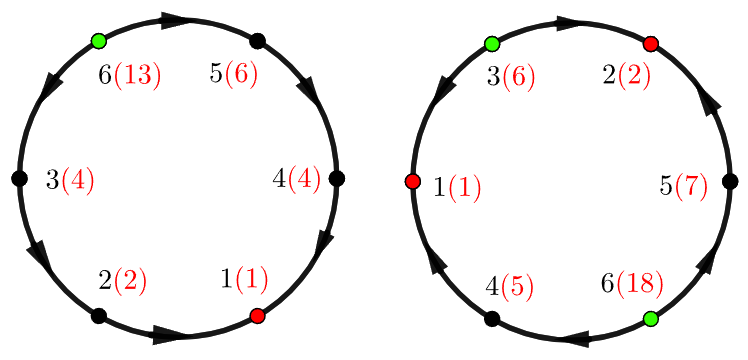}
    \caption{Oriented $C_6$s that are $\{0,2,3\}$-antimagic and $\{0,1,2\}$-antimagic, respectively.}
    \label{fig:C6 D>2}
\end{figure}

Figure \ref{fig:C6 D>2} provides examples of $D$-antimagic oriented cycles, where $|D| \geq 2$. Therefore, we conclude this section by posing the following general problem:     
\begin{problem}
For $n\geq 3$ and $|D| \ge 3$, characterize $D$-antimagic oriented cycles.\end{problem}


\section{$D$-antimagic $2$-regular graphs} \label{sec:2reg}

In this section, we investigate $D$-antimagic labelings of oriented $2$-regular graphs containing at least two cycle components. Similar to oriented cycles in Section \ref{sec:cycle}, we only consider two orientations: unidirectional and $\Theta$.

For 
$1 \leq j \leq t$, $m_j \geq 1$, $n_j \geq 3$, and $\sum_{j=1}^tm_j\ge2$, we use the following notation for an oriented 2-regular graph where the cycle components are arranged nondecreasingly: $$\overrightarrow{C} = \bigcup_{j=1}^t m_j \overrightarrow{C_{n_j}}, \text{ where } n_{j_1} < n_{j_2} \text{ if } j_1 < j_2.$$ 

An oriented 2-regular graph is considered to be \textit{unidirectional} if each cycle component is unidirectional. Similarly, \(\overrightarrow{C}\) is \textit{$\Theta$-oriented} if each cycle component is $\Theta$-oriented. 
We start by characterizing $\{1\}$-antimagic oriented 2-regular graphs.

\begin{theorem}
An oriented 2-regular graph is $\{1\}$-antimagic if and only if at most one of its cycle components is $\Theta$-oriented, while the rest are unidirectional.   
\end{theorem}
\begin{proof}
If each cycle of $\overrightarrow{C}$ is unidirectional, then each vertex has a unique and distinct out-neighbor, resulting in distinct $\{1\}$-weights. 

Now, consider the case where there is a $\Theta$-oriented cycle component of order $\nu$. Denote its vertices by $u_l$, $1\le l\le \nu$, where $u_1$ and $u_{\nu}$ are the source and the sink, respectively. Let $f_0$ be a vertex labeling of all the unidirectional cycle components. Define a bijection $f_*$ where $f_*(u_l)=\sum_{j=1}^t m_jn_j-l+1$ and $f_*(v_{i,s}^j)=f_0(v_{i,s}^j)$.  Then, each vertex of $\overrightarrow{C}$, other than $u_{1}$, has a unique out-neighbor. This leads to a distinct $\{1\}$-weight for each vertex, which is less than $\omega_{\{1\}}(u_{1})=2\sum_{j=1}^t m_jn_j-\nu$.

Conversely, recall that each $\Theta$-oriented cycle has exactly one sink. If the oriented 2-regular graph has two or more $\Theta$-oriented cycle components, there are at least two vertices with zero $\{1\}$-weights. Hence, the graph is not $\{1\}$-antimagic.
\end{proof}

The next theorem completes the characterization of $D$-antimagic oriented 2-regular graphs, where $D$ consists of a singleton. 

\begin{theorem}
For \(2 \leq k \leq n_1-1\), \(\bigcup_{j=1}^t m_j \overrightarrow{C_{n_j}}\) is \(\{k\}\)-antimagic if and only if it is unidirectional.
\end{theorem}
\begin{proof}
If $\overrightarrow{C}=\bigcup_{j=1}^t m_j \overrightarrow{C_{n_j}}$ is not unidirectional, then it has at least one sink with at least one in-neighbor. For any bijection, the $\{k\}$-weights of those two vertices are zero.  

If $\overrightarrow{C}$ is unidirectional then each vertex has a unique out-neighbor, and so any bijection is a $\{k\}$-antimagic labeling of $\overrightarrow{C}$. 
\end{proof}

The next step is to characterize $D$-antimagic oriented 2-regular graphs, where the cardinality of $D$ is at least 2, as proposed in the following:
\begin{problem}
For $|D|\ge 2$, characterize $D$-antimagic oriented 2-regular graphs.
\end{problem}

Our last result is a construction of a $D$-antimagic labeling for a $\Theta$-oriented 2-regular graph. 

\begin{theorem}
A $\Theta$-oriented 2-regular graph is \(D\)-antimagic if and only if $\min(D)=0$.
\end{theorem}
\begin{proof} Since each cycle component has one sink, if $\min(D)>0$, those sinks will share zero $D$-weight.

For $t\geq 2$, $1\le j\le t, 1\le s\le m_j,$ and $1\le i\le n_j$, define a bijection $h_*:V\left(\overrightarrow{C}\right) \rightarrow \{1, 2, \dots,\sum_{p=1}^t m_pn_p\}$ by 
\begin{align*}
    h_*\left(v_i^{j,s}\right)=&\sum_{q=1}^{j_i-1}\sum_{p=q}^t m_p\left(n_q-n_{q-1}\right)+\left(i-n_{j_i-1}-1\right)\sum_{q=j_i}^t m_q\\&+\sum_{q=j_i}^{j-1} m_q +s,
\end{align*}
where $n_0=0$ and $j_i$ is a natural number such that $n_{j_i-1}+1\le i\le n_{j_i}$.

Let $D=\{d_0,d_1,\dots,d_p\}$, where $0=d_0<d_1<\dots<d_p\le n-2$ and $p_{n_j}=\max\{r\in [1,p]\mid d_r\le n_j\}$. We consider the following two cases:

\noindent \textbf{Case 1: $d_1=1$.} The $D$-neighborhood and $D$-weight of each vertex are:

\noindent For $d_{r}+1\le i\le d_{r+1},0\le r\le p_{n_j}-1$: $N_D(v_{i,s}^j)=\{v_{i-d_l,s}^j|0 \le l \le r\}$ and 
$\omega_D(v_{i,s}^j)=\sum_{l=0}^r\Biggl(\sum_{q=1}^{j_{(i-d_l)}-1}\sum_{p=q}^t m_p(n_q-n_{q-1})+\sum_{q=j_{(i-d_l)}}^{j-1} m_q +s+\left((i-d_l)-n_{j_{(i-d_l)}-1}-1\right)\sum_{q=j_{(i-d_l)}}^t m_q  \Biggr)$.
 
\noindent For $d_{p_{n_j}}+1\le i\le n_j-1$: $N_D(v_{i,s}^j)= \{v_{i-d_l,s}^j|0 \le l \le p_{n_j}\}$ and $\omega_D(v_{i,s}^j)=\sum_{l=0}^{p_{n_j}}\Biggl(\sum_{q=1}^{j_{(i-d_l)}-1}\sum_{p=q}^t m_p(n_q-n_{q-1})+\sum_{q=j_{(i-d_l)}}^{j-1} m_q +s+\left((i-d_l)-n_{j_{(i-d_l)}-1}-1\right)\sum_{q=j_{(i-d_l)}}^t m_q  \Biggr)$.
    
\noindent For $i=n_j$: $N_D(v_{n_j,s}^j)= \{v_{n-d_l,s}^j|0 \le l \le p_{n_j}\}\cup\{v_{1,s}^j\}$ and $\omega_D(v_{i,s}^j)=\sum_{l=0}^{p_{n_j}}\Biggl(\sum_{q=1}^{j_{(i-d_l)}-1}\sum_{p=q}^t m_p(n_q-n_{q-1})  +\sum_{q=j_{(i-d_l)}}^{j-1} m_q +s+\left((i-d_l)-n_{j_{(i-d_l)}-1}-1\right)\sum_{q=j_{(i-d_l)}}^t m_q\Biggr)+\sum_{l=1}^{j-1}n_j+s$.

\noindent \textbf{Case 2: $d_1\ne 1$.} The $D$-neighborhoods and $D$-weights for all vertices are:
    
\noindent For $d_{r}+1\le i\le d_{r+1},0\le r\le p_{n_j}-1$: $N_D(v_{i,s}^j)=\{v_{i-d_l,s}^j|0 \le l \le r\}$ and
$\omega_D(v_{i,s}^j)=\sum_{l=0}^r\Biggl(\sum_{q=1}^{j_{(i-d_l)}-1}\sum_{p=q}^t m_p(n_q-n_{q-1})+\sum_{q=j_{(i-d_l)}}^{j-1} m_q +s+\left((i-d_l)-n_{j_{(i-d_l)}-1}-1\right)\sum_{q=j_{(i-d_l)}}^t m_q  \Biggr)$.

\noindent For $d_{p_{n_j}}+1\le i\le n_j$: $N_D(v_{i,s}^j)= \{v_{i-d_l,s}^j|0 \le l \le p_{n_j}\}$ and 
$\omega_D(v_{i,s}^j)=\sum_{l=0}^{p_{n_j}}\Biggl(\sum_{q=1}^{j_{(i-d_l)}-1}\sum_{p=q}^t m_p(n_q-n_{q-1})+\sum_{q=j_{(i-d_l)}}^{j-1} m_q +s+\left((i-d_l)-n_{j_{(i-d_l)}-1}-1\right)\sum_{q=j_{(i-d_l)}}^t m_q  \Biggr)$.

Now we shall show that the $D$-weights are distinct in both cases by considering the following:
\begin{description}
\item[Case A: $i\ne i'$.] Without loss of generality, if $i>i'$ then $\omega_D(v_{i,s}^{j})>\omega_D(v_{i',s'}^{j'})$.
\item[Case B: $i=i'$.] Consider the following two subcases based on $j$:
\begin{description}
\item[Subcase B.1: $j\ne j'$.] Without loss of generality, if $j>j'$ then $\omega_D(v_{i,s}^{j})>\omega_D(v_{i,s'}^{j'})$.
\item[Subcase B.2: $j= j'$.] If $s>s'$ then $\omega_D(v_{i,s}^{j})>\omega_D(v_{i,s'}^{j})$.
\end{description}
\end{description}

Figure \ref{fig:C} provides an example of the labeling.\end{proof}

\begin{figure}[ht]
    \centering
    \includegraphics[width=1
\linewidth]{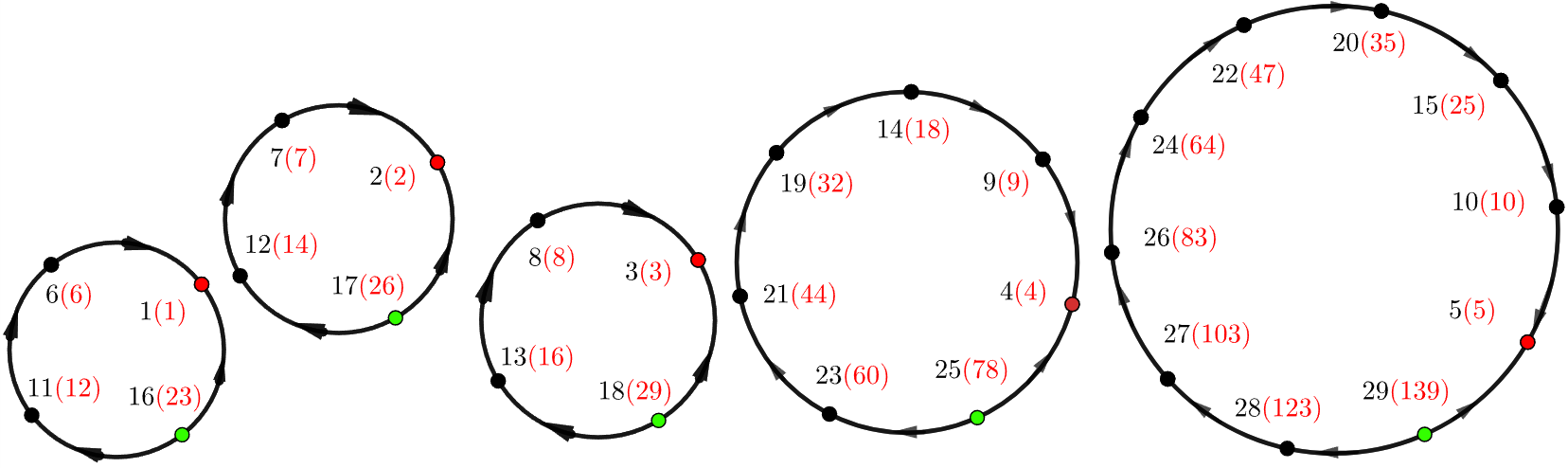}
    \caption{A $\{0,2,5,6,7\}$-antimagic labeling of $\Theta$-oriented $3\overrightarrow{C_4}\cup \overrightarrow{C_7}\cup\overrightarrow{C_{10}}$.}
    \label{fig:C}
\end{figure}

We conclude by proposing the following general open problem.

\begin{problem}
Find all possible orientations of a 2-regular graph admitting a $D$-antimagic labeling.
\end{problem}




\begin{thebibliography}{99}

\bibitem{Abrarlinearforest} A. M. Abrar and R. Simanjuntak, "$D$-antimagic labeling of oriented linear forest", submitted.
\bibitem{Handa} A. K. Handa, A. Godinho, T. Singh, and S. Arumugam, "Distance antimagic labeling of join and corona of two graphs," \textit{AKCE International Journal of Graphs and Combinatorics}, vol. 14, pp. 172--177, 2017.
\bibitem{Kamatchi-64-13} N. Kamatchi and S. Arumugam, "Distance antimagic graphs," \textit{Journal of Combinatorial Mathematics and Combinatorial Computing}, vol. 64, pp. 61--67, 2013.
\bibitem{Kamatchi cube} N. Kamatchi, G. R. Vijayakumar, A. Ramalakshmi, S. Nilavarasi, and S. Arumugam, "Distance antimagic labelings of graphs," \textit{Lecture Notes in Computer Science}, vol. 10398, pp. 113--118, 2017.
\bibitem{AAC} R. Simanjuntak, T. Nadeak, F. Yasin, K. Wijaya, N. Hinding, and K. A. Sugeng, "Another antimagic conjecture", \textit{Symmetry}, vol. 13, no. 12, p. 2071, 2021.
\bibitem{simanjuntakprod} R. Simanjuntak and A. Tritama, "Distance antimagic product graphs", \textit{Symmetry}, vol. 14, no. 8, p. 1411, 2022.
\bibitem{sy2014distance} S. Sy, R. Simanjuntak, T. Nadeak, K. A. Sugeng, and T. Tulus, "Distance antimagic labeling of circulant graphs", \textit{AIMS Mathematics}, vol. 9, pp. 21177--21188, 2024.
\bibitem{DB West} D. B. West, \textit{Introduction to Graph Theory}, 2nd ed. Prentice Hall, 2001, p. 26, ISBN: 9780132278287.
\end{thebibliography}
\end{document}